\documentclass[a4paper,reqno,11pt]{amsart}
\usepackage[utf8]{inputenc}	
\usepackage{indentfirst} %Per avere il rientro nella prima riga di ogni capitolo.
\usepackage{amsmath,amsfonts,amssymb,,mathrsfs}
\usepackage{bbm}
\usepackage[colorlinks=true]{hyperref}
\usepackage{enumerate}
\usepackage{graphicx}
\usepackage{xcolor}
\usepackage{stackengine}
\usepackage[a4paper, top=2.5cm, right=2.5cm, left=2.5cm, bottom=2.5cm]{geometry}
\usepackage{esint}
\usepackage[shortlabels]{enumitem}

\allowdisplaybreaks

\theoremstyle{definition}
\newtheorem{definition}{Definition}[section]

\theoremstyle{remark}
\newtheorem{remark}[definition]{Remark}

\theoremstyle{plain}
\newtheorem{theorem}[definition]{Theorem}

\theoremstyle{plain}
\newtheorem{lemma}[definition]{Lemma}
\theoremstyle{remark}

\theoremstyle{plain}

\theoremstyle{plain}
\newtheorem{proposition}[definition]{Proposition}
\theoremstyle{plain}

\theoremstyle{plain}

\numberwithin{equation}{section}
\allowdisplaybreaks

\def\XXint#1#2#3{{\setbox0=\hbox{$#1{#2#3}{\int}$}
      \vcenter{\hbox{$#2#3$}}\kern-.5\wd0}}

\definecolor{ao}{rgb}{0.0, 0.5, 0.0}

\newcommand{\Om}{\Omega}
\newcommand{\R}{\mathbb{R}}

\newcommand{\di}{\mathrm{d}}

\newcommand{\de}{\textnormal{d}}

\newcommand{\N}{\mathbb{N}}
\newcommand{\Q}{\mathbb{Q}}
\newcommand{\Z}{\mathbb{Z}}

\def\namedlabel#1#2{\begingroup
    #2%
    \def\@currentlabel{#2}%
    \phantomsection\label{#1}\endgroup
}

\makeatletter
\@namedef{subjclassname@2020}{%
  \textup{2020} Mathematics Subject Classification}
\makeatother

\title[Symmetric total variation on point clouds]{Approximation of symmetric total variation\\ on point clouds}

\author[S. Almi]{S. Almi}
\address[Stefano Almi]{Department of Mathematics and Applications ``R.~Caccioppoli'', University of Naples Federico II, Via Cintia, Monte S. Angelo, 80126 Napoli, Italy.}
\email{stefano.almi@unina.it}
\author[A. Kubin]{A. Kubin}
\address[Anna Kubin]{Institute of Analysis and Scientific Computing, TU Wien, Wiedner Hauptstr.~8-10, 1040 Vienna, Austria.}
\email{anna.kubin@tuwien.ac.at}
\author[E. Tasso]{E. Tasso}
\address[Emanuele Tasso]{Institute of Analysis and Scientific Computing, TU Wien, Wiedner Hauptstr.~8-10, 1040 Vienna, Austria.}
\email{emanuele.tasso@tuwien.ac.at}

\begin{document}

\begin{abstract}
	   The paper investigates the approximation of the symmetric Total Variation functional on graphs. Such an approximation is given in terms of a discrete and symmetric finite difference model defined on point clouds obtained by randomly sampling a reference probability measure. We identify suitable scalings of the point distribution that guarantee an almost surely $\Gamma$-convergence to an anisotropic weighted symmetric Total Variation.

		\medskip
  

\end{abstract}

 \maketitle

\section{Introduction}

The analysis of variational problems defined on random data has become increasingly significant in a wide range of applications, including machine learning, imaging, and materials science \cite{Bressonetal, GarciaSellesetal, Rangapuram, Shi, Sunetal}.
To a point cloud, one can naturally associate a weighted graph:  the sampled points form the vertices of the graph, and edges are introduced between pairs of points that are sufficiently close to each other. These interactions are encoded by weights that depend on the distance between points through a kernel function with a prescribed interaction length scale.
The choice of this length scale plays a crucial role. On the one hand, reducing the number of edges is desirable in order to lower the computational complexity. On the other hand, if the distance between the nodes falls below a certain threshold, the graph may fail to capture the relevant geometric features of the underlying point cloud. 

In the general context outlined above, many machine learning tasks are then formulated as a minimization problem of a functional defined on the graph representing the data set. The consistency of such problems as the number of samples tends to infinity becomes a fundamental question. In a  mathematical perspective, one aims at showing the convergence of discrete variational models posed on random point clouds towards their continuum counterparts. 
This issue has been addressed in the scalar setting in several different settings: for the total variation and perimeter functionals in \cite{Trillos-Slepcev, Trillos2020, Cristoferi-Thorpe}, for the Ginzburg-Landau functional in~\cite{Thorpe-Theil}, for the Dirichlet energy in~\cite{Slepcev-Thorpe}, and in \cite{Caroccia-Chambolle-Slepcev} for the Mumford--Shah functional. In such works, a point cloud is modeled by a set of random points $\{X_1,\ldots, X_n\}$ obtained by sampling a given probability distribution $\nu = \rho \, \de x$ in a bounded domain $D \subset \R^d$. The points are assumed to interact at a given scale $\varepsilon_{n}>0$, which converges to $0$ with a suitable rate depending on the ambient dimension~$d$. Such a rate is dictated by the use of a transport-like distance $TL^{p}$ for $p \in [1, +\infty)$ (see Definition~\ref{d:TL1} and, e.g.,~\cite{Trillos-Slepcev, GTS}), which simultaneously describes the Wasserstein convergence of the empirical measure $\nu_{n} := \frac{1}{n} \sum_{i} \delta_{X_{i}}$ towards~$\nu$ and the convergence of $L^{p}$-functions on graphs towards targets in $L^{p} (D; \nu)$. We further refer to~\cite{Braides-Caroccia, Braides-Piatnitski, Caroccia} for results of compactness, $\Gamma$-convergence, and homogenization on Poisson point clouds, where the use of the $TL^{p}$-distance is not permitted.

Recent works (see \cite{GarciaSellesetal, Sunetal}) suggest the relevance of studying variational models from fracture mechanics  on point cloud data. In particular, these contributions indicate that geometric information extracted from large-scale vectorial point clouds can be used to identify and characterize fracture systems, therefore motivating the formulation of variational fracture models, such as the Griffith functional~\cite{Bourdinetal}, in the discrete setting. 
As a first step toward extending \cite{Trillos-Slepcev, Caroccia-Chambolle-Slepcev} to the vectorial setting, we consider a discrete approximation of the \emph{vectorial symmetric total variation} defined over a set of random points, which can be viewed as a simplified variational model of linear elasticity.

For a reference probability measure $\nu = \rho \, \de x$ absolutely continuous with respect to the Lebesgue measure in~$\R^{d}$ with density $\rho$ having support in an open bounded subset $D$ of~$\R^{d}$, we assume $\rho$ to be continuous, bounded, and bounded away from~$0$ (see also $(\rho1)$--$(\rho2)$ below). We consider $\{X_1,\ldots, X_n\}$ random points i.i.d.~as $\nu$ and fix an interaction length-scale  $\varepsilon_{n}>0$, which determines the neighbourhood within which pairs of points are allowed to interact. 
Given a non-negative, radially symmetric kernel $\eta \colon \R^d \to [0,\infty)$ (see Section \ref{s:2} for the precise assumptions on $\eta$), we define its rescaling at scale $\varepsilon_{n}$ by 
\[
\eta_{\varepsilon_{n}}  :=\frac{1}{\varepsilon_{n}^d}\eta\left(\frac{\cdot}{\varepsilon}\right).
\]
For a vector-valued function $u\colon \{X_1,\ldots, X_n\} \to \R^d$,
the \emph{graph vectorial symmetric total variation} is defined as
\begin{equation}\label{funcloud}
    GTV_{n,\varepsilon_{n}}(u):= \frac{1}{\varepsilon_{n}^2}\frac{1}{n^2} \sum_{i,j=1}^{n} \eta_{\varepsilon_{n}} (X_i-X_j) |(u(X_i)-u(X_j)) \cdot(X_i-X_j)|.
\end{equation}
 The normalization factor $1/n^2$ averages the contributions over all interacting pairs of points. The factor $1/\varepsilon_{n}^2$ provides the correct scaling with respect to the interaction length scale. 
  Indeed, due to the presence of the rescaled kernel $\eta_{\varepsilon_{n}}$,
  only pairs of points at a distance of order $\varepsilon_{n}$ contribute significantly to the sum.
  For such pairs, the increment $u(X_i)-u(X_j)$ is typically of order $\varepsilon_{n}$ if $u$ is smooth. Hence, its projection onto the edge $X_i-X_j$ is of order $\varepsilon_{n}^2$. The prefactor $1/\varepsilon_{n}^2$  compensates for this scaling, thus
  ensuring consistency with the corresponding continuum symmetric total variation in the limit $\varepsilon_{n} \to 0$. Compared to~\cite{Trillos-Slepcev, Caroccia-Chambolle-Slepcev}, the extra projection on $X_{i} - X_{j}$ represents the key change in the graph-functional, as $GTV_{n, \varepsilon_{n}}$ only involves symmetric finite differences, and therefore leads to a $BD$-total variation functional.

  Our main result consists in showing that, under suitable assumptions on the kernel and the distribution of the point clouds, the functionals in \eqref{funcloud} almost surely $\Gamma$-converge in the $TL^1$-topology to a functional which only depends on a weighted and anisotropic symmetric total variation. More precisely we have the following theorem. We refer to Section~\ref{s:2} for the set of assumptions.

\begin{theorem}
\label{thm:main}
Assume $({\rm K1})$--$({\rm K3})$, $(\rho1)$--$(\rho2)$, and~\eqref{eq:asssumption_epsilon}--\eqref{e:second-Tn}. Then, the sequence of functionals ~$GTV_{n}$ almost surely $\Gamma$-converges with respect to the $TL^{1}$-convergence to
\begin{equation*}
    TV_{\eta}(u;\rho^2):=\int_{D }\rho^{2}(x)\phi_{\eta}\bigg ( \frac{Eu(x)}{|Eu(x)|}\bigg ) \,  \de |Eu(x)|,
    \end{equation*}
    where the function $\phi_{\eta} \colon \mathbb{M}^{d}_{sym} \to [0, + \infty)$ is defined as $\phi_{\eta}(A):= \int_{\R^d} \eta(\xi) |A\xi \cdot \xi|\, \de \xi$.
\end{theorem}

    As in~\cite{Trillos-Slepcev}, our analysis relies the existence of a transport map $T_{n}$ between $\nu$ and $\nu_{n}$ such that 
    \begin{align*}
        \lim_{n\to \infty} \frac{ n^{1/d} \| Id - T_{n}\|_{L^{\infty}}}{\log^{1/d} (n) } = 0\,,
    \end{align*}
    where $Id\colon D \to D$ denotes the identity function. Such a map was proven to exist in~\cite{GTS, Caroccia-Chambolle-Slepcev} and allows to rewrite the functional~\eqref{funcloud} in a continuum setting as
  \begin{align}
  \label{e:intro1}
      GTV_{n,\varepsilon_{n}}(u) = \frac{1}{\varepsilon_n^2}  \iint_{D \times D}  \eta_{\varepsilon_n}(T_n(x)-T_n(y)) \big | &(u \circ T_n (x)-u\circ T_n(y)) \\
      &\cdot (T_n(x)-T_n(y))\big | \rho(x) \rho(y)\, \de x \de y. \nonumber
  \end{align}
  Notice that, even defining $v := u \circ T_{n}$ as a function over $D$ and exploiting the monotonicity properties of~$\eta$ to reduce to work with~$\eta_{\varepsilon_{n}} (x - y)$, we are not in a position to treat the right-hand side of~\eqref{e:intro1} as an auxiliary nonlocal energy defined over $L^{1} (D;\R^{d})$, as it was done in~\cite{Trillos-Slepcev}, as the transport map~$T_{n}$ still appears in the projection part of the integrand. Under the scaling assumption
  \begin{align}
  \label{e:intro2}
      \limsup_{n \to \infty} \, \frac{ \log^{1/d} (n) }{n^{1/d}} \, \frac{1}{\varepsilon^{2}_{n}} < +\infty\,,
  \end{align}
  we are able to perform a slicing argument similar to that of~\cite{Gobbino, Gobbino-Mora, AlmiDavoliKubinTasso}. We point out that condition~\eqref{e:intro2} is stronger than the scaling considered in the $BV$ setting of~\cite{Trillos-Slepcev, Caroccia-Chambolle-Slepcev}, as it implies a uniform control on a discrete second order derivative of~$T_n$ (see \eqref{e:second-Tn} below). From a geometric point of view,~\eqref{e:intro2} means that the cut-off function~$\eta_{\varepsilon}$ in~\eqref{e:intro1} has to weight a larger number of points in order to reconstruct the $BD$-total variation in the limit. A similar phenomenon is encountered in other discrete finite-difference models approximating free-discontinuity functionals involving $BD$-type functions. For instance, we refer to the Ambrosio-Tortorelli approximation on square lattices in~\cite{Crismale-Scilla-Solombrino}, where the authors have to consider next-to-next-nearest neighbours interactions, which are instead not necessary in the $BV$-setting~\cite{BachBraidesZeppieri}. Hypothesis~\eqref{e:intro2} is used in the $\Gamma$-liminf inequality (see Theorem~\ref{thm:liminf}) as well as to ensure that for a sequence $(u_{n}, \nu_{n})$ converging to $(u, \nu)$ in the $TL^{1}$-metric the limit map $u$ belongs to $BD(D)$. In this respect, we notice that, by construction, the~$\varepsilon_{n}$-discrete second order derivative of~$T_{n}$ converges to $0$ in the sense of distributions. Condition~\eqref{e:intro2} improves it to a  weak$^{*}$ convergence in $L^{\infty}(D; \R^{d})$, which is in duality with the $L^{1}$-convergence of~$u_{n} \circ T_{n}$, implied by the $TL^{1}$-convergence.

  An interesting research line would be to understand the $\Gamma$-convergence of $GTV_{n, \varepsilon_{n}}$ in the scaling $\frac{(\log n)^{1/d}}{n^{1/d}} \ll \varepsilon_{n} \ll \frac{(\log n)^{1/2d}}{n^{1/2d}}$, thus recovering the scalings of~\cite{Trillos-Slepcev, Caroccia-Chambolle-Slepcev}. To do this, one may have to investigate how the geometrical and topological properties of the graph influence the regularity of the transport maps~$T_{n}$ constructed in~\cite{GTS}.

\section{Energy and assumptions}
\label{s:2}

Let $\eta \colon \R^d \to [0, +\infty)$ be a radially symmetric kernel, $\eta(x):= \boldsymbol{\eta}(|x|)$ where $\boldsymbol{\eta} \colon [0, + \infty) \to [0, + \infty)$ is such that
\begin{itemize}
    \item[(K1)] $\boldsymbol{\eta}(0)>0$ and $\boldsymbol{\eta}$ is continuous at $0$;
    \item[(K2)] $\boldsymbol{\eta}$ is non-increasing;
    \item[(K3)] the integral $\int_{\R^d} \eta(x) |x|^2 \,\de x $ is finite.
\end{itemize}
For $x \in \R^{d}$ and $\varepsilon>0$ we define
\[
\eta_\varepsilon (x) :=\frac{1}{\varepsilon^d}\eta\left(\frac{x}{\varepsilon}\right).
\]

Let $(\Omega, \mathbb{P}, \mathcal{F})$ be a probability space, let $D \subset \R^d$ be a bounded open set with Lipschitz boundary, and $\rho\colon  D \to \R$ be such that
\begin{itemize}
    \item[$(\rho1)$] $\rho$ is continuous;
    \item[$(\rho2)$] there exist $0 < \alpha \leq \beta < +\infty$ such that $\alpha \leq \rho(x) \leq \beta$ for every $x \in D$.
\end{itemize} 
We define $\nu := \rho \mathcal{L}^d$ and let $X_1,\ldots, X_n \colon \Omega \to D$ be $n$ random points i.i.d. according to~$\nu$. Let $\nu_n$ be the empirical measure associated with the $n$ data points, i.e.,
\[
\nu_n := \frac{1}{n} \sum_{i=1}^n \delta_{X_i}.
\]
Notice that $\nu_{n} \in \mathcal{P} (D)$ is itself a random variable. If not explicitly needed, we will not indicate the dependence on the realization $\omega \in \Omega$, as our analysis holds almost surely (i.e., for $\mathbb{P}$-a.e.~$\omega \in \Omega$). For every $u \colon \{ X_1,\ldots, X_n\} \to \R^d$ we define the \emph{graph vectorial symmetric total variation} by
\begin{equation}
    GTV_{n,\varepsilon}(u):= \frac{1}{\varepsilon^2}\frac{1}{n^2} \sum_{i,j} \eta_\varepsilon(X_i-X_j) |(u(X_i)-u(X_j)) \cdot(X_i-X_j)|.
\end{equation}

We fix a sequence $\varepsilon_{n} >0$ such that
\begin{equation}
\label{eq:asssumption_epsilon}
    \limsup_{n \to \infty} \frac{(\log n)^{1/d}}{n^{1/d}} \frac{1}{\varepsilon^2_n} <+\infty\,.
\end{equation}
In \cite{GTS, Caroccia-Chambolle-Slepcev} it was shown that, for $\mathbb{P}$-a.e.~$\omega \in \Omega$, there exist $C>0$ and  a sequence of transportation maps $\{T_n\}_{n \in \mathbb{N}}$ such that $(T_n)_\# \nu = \nu_n$ and
\begin{align}
\label{e:Tn-scaling}
    \lim_{n \to \infty} \frac{n^{1/d}\|Id-T_n\|_{L^\infty}}{(\log n)^{1/d}} \le C.
\end{align}
We notice that if $\varepsilon_n$ and $T_n$ satisfy~\eqref{eq:asssumption_epsilon}--\eqref{e:Tn-scaling}, then it holds
\begin{align}
\label{eq:id-tntozero}
    &\lim_{n \to \infty} \frac{\|Id -T_n\|_{L^\infty}}{\varepsilon_n} = 0\,,
    \\
\label{e:second-Tn}
    &\limsup_{n \to \infty} \frac{\|T_n(\cdot+\varepsilon_n)-2T_n(\cdot) + T_n(\cdot-\varepsilon_n)\|_{L^\infty}}{\varepsilon^2_n} <+\infty \,.
\end{align}

Given $T_{n}$ as above, we write
\begin{align}
\label{eq:GTV}
&GTV_{n}(u):=GTV_{n, \varepsilon_n}(u)\\
&=\frac{1}{\varepsilon_n^2} \iint_{D \times D} \eta_{\varepsilon_n}(T_n(x)-T_n(y)) \big |(u \circ T_n (x)-u\circ T_n(y)) \cdot (T_n(x)-T_n(y))\big | \rho(x) \rho(y)\, \de x \de y. \nonumber
\end{align}

For $u \in BD(D)$ we define
\begin{equation*}
    TV_{\eta}(u;\rho^2):=\int_{D }\rho(x)^2\phi_{\eta}\bigg ( \frac{Eu(x)}{|Eu(x)|}\bigg ) \,  \de |Eu(x)|,
\end{equation*}
where we have introduces the norm on symmetric  matrices $\phi_{\eta} \colon \mathbb{M}^{d}_{sym} \to [0, + \infty)$ as
    \[
    \phi_{\eta}(A):= \int_{\R^d} \eta(\xi) |A\xi \cdot \xi|\, \de \xi.
    \]

In the following definition we recall the $TL^{1}$-convergence introduced in~\cite{Trillos-Slepcev}.

\begin{definition}
\label{d:TL1}
Let $\mu_{1}, \mu_{2} \in \mathcal{P} (D)$, $w_{1} \in L^{1} (D; \R^{d}; \mu_{1})$ and $w_{2} \in L^{1} (D; \R^{d}; \mu_{2})$. We define the $TL^{1}$-distance as
\begin{align*}
\de_{TL^{1}} \big( ( w_{1}, \mu_{1}) , (w_{2} , \mu_{2} ) \big) := \inf_{\gamma \in \Gamma(\mu_{1}, \mu_{2})} \iint_{D \times D}| x - y | +  | w_{1} (x) - w_{2} (y) | \, \de \gamma(x, y)\,,
\end{align*} 
where $\Gamma (\mu_{1}, \mu_{2})$ denotes the set of transport plans between $\mu_{1}$ and $\mu_{2}$. For $\mu_{n}, \mu \in \mathcal{P} (\R^{d})$, $w_{n} \in L^{1} (\Om; \R^{d}; \mu_{n})$, and $w \in L^{1} (D; \R^{d}; \mu)$, we say that $(w_{n}, \mu_{n}) \to (w, \mu)$ in the $TL^{1}$-metric if
\begin{align*}
\lim_{n \to \infty} \, \de_{TL^{1}} \big( (w_{n}, \mu_{n}) , (w, \mu)\big) = 0\,.
\end{align*}
\end{definition}

The following characterization can be found in~\cite[Proposition 3.12]{Trillos-Slepcev}.

\begin{proposition}
\label{p:TL1-convergence}
Let $\mu_{n}, \mu \in \mathcal{P} (\R^{d})$, $w_{n} \in L^{1} (D; \R^{d}; \mu_{n})$, and $w \in L^{1} (D; \R^{d}; \mu)$, and assume that $\mu \ll \mathcal{L}^{d}$. Then, the following are equivalent:
\begin{itemize}
\item[(1)] $(w_{n}, \mu_{n}) \to (w, \mu)$ in the $TL^{1}$-metric;

\item[(2)] for every transport map $S_{n} \colon D \to D$ (i.e., such that $(S_{n})_{\#} \mu = \mu_{n}$) such that
\begin{align*}
\lim_{n \to \infty} \int_{D} | x - S_{n} (x) | \, \de \mu(x) = 0
\end{align*} 
we have that
\begin{align*}
\lim_{n \to \infty} \int_{D} | w(x)-w_{n} (S_{n} (x))| \, \de \mu(x) = 0\,.
\end{align*}
\end{itemize}
\end{proposition}

The proof of Theorem~\ref{thm:main} is carried out in the next two sections, where we prove the $\Gamma$-liminf and the $\Gamma$-limsup inequalities, respectively.

\section{Gamma liminf inequality}

In this section we establish the liminf inequality for the functionals $GTV_{n}$. 

\begin{theorem}\label{thm:liminf}
    Assume $({\rm K1})$--$({\rm K3})$, $(\rho1)$--$(\rho2)$, and~\eqref{eq:asssumption_epsilon}. For $\mathbb{P}$-a.e.~$\omega \in \Omega$, for every 
    $u \in L^1(D;\R^d;\nu)$ and $\{u_{n}\}_{n \in \N} \subset L^1(\{X_1,\ldots, X_{n}\};\R^d;\nu_{n})$ such that $(\nu_{n},u_{n}) \to (\nu,u)$ in $TL^1$, there exists a subsequence $\varepsilon_n \to 0$ such that
    \begin{equation*}
        \liminf_{n \to \infty} GTV_{n}(u_{n}) \ge  TV_{\eta}(u;\rho^2)\,.
    \end{equation*}
    In particular, $u \in BD(D)$. 
\end{theorem}

\begin{remark}\label{rmk:1}
    As in \cite[Section~5]{Trillos-Slepcev}, we may assume that the kernel $\eta$ is of the form $\eta(t)=c$ for $t<b$ and $\eta(t)=0$ for $t\ge b$. Indeed, if we can establish the liminf inequality under this assumption, then, by the superadditivity of the liminf, the same inequality also holds for functions of the form $\eta=\sum_{k=1}^l \eta_k$, for some $l\in\mathbb{N}$, where each $\eta_k$ is of the above form. Finally, the case of general $\eta$ follows by considering an increasing sequence of piecewise constant functions $\eta_n \colon [0, +\infty)\to [0, + \infty)$ such that $\eta_n \nearrow \eta$ almost everywhere, and by the continuity of the map $\eta \mapsto TV_{\eta}$.
\end{remark}

%We first argue as in \cite[Section $5$]{Trillos-Slepcev}. 

By \cite{GTS, Caroccia-Chambolle-Slepcev}, we consider $\omega \in \Omega$ and $T_{n}\colon D \to D$ such that~\eqref{e:Tn-scaling} (and thus~\eqref{eq:id-tntozero}--\eqref{e:second-Tn}) holds. In view of Remark \ref{rmk:1}, we assume for the remaining part of this section that $\eta$ is of the form $\eta(t)=c$ for $t< b$ and $\eta(t)=0$ for $t\ge b$.
For almost every $(x,y) \in D \times D$ we have
\begin{equation}\label{eq:implicaT_n}
    |T_n(x)-T_n(y)|>b \varepsilon_n \Rightarrow |x-y|>b \varepsilon_n -2 \|Id-T_n\|_{L^\infty}.
\end{equation}

Thanks to \eqref{eq:id-tntozero}, for $n$ large enough it holds
\[
\tilde \varepsilon_n :=\varepsilon_n - \frac{2}{b} \|Id-T_n\|_{L^\infty} >0.
\]
By \eqref{eq:implicaT_n} and $({\rm K2})$, for $n$ large enough and for almost every $(x,y) \in D \times D$ we get
\[
\eta \left ( \frac{|x-y|}{\tilde \varepsilon_n}\right ) \le \eta \left (\frac{|T_n(x)-T_n(y)|}{\varepsilon_n} \right ).
\]
Hence, we have the lower bound
\begin{align*}
    &GTV_{n}(u_n)=GTV_{n,\varepsilon_n}(u_n) \\
    &\ge \left (\frac{\tilde \varepsilon_n}{\varepsilon_n} \right)^{d+2} \frac{1}{\tilde \varepsilon_n^2}\iint_{{}_{\scriptstyle D \times D}} \! \frac{1}{\tilde\varepsilon_n^d}\eta \left (\frac{|x-y|}{\tilde \varepsilon_n} \right ) |(u_n\circ T_n(x)-u_n\circ T_n(y)){\, \cdot \,} (T_n(x)-T_n(y))|\rho(x) \rho(y)\di x \di y .
\end{align*}

Since $\frac{\tilde \varepsilon_n}{\varepsilon_n} \to 1$, it is enough to prove the liminf estimate for
\begin{equation*}
    \frac{1}{\tilde \varepsilon_n^2} \iint_{D \times D} \eta_{\tilde \varepsilon_n}(x-y) \big |(u_n \circ T_n (x)-u_n\circ T_n(y)) {\, \cdot \,}  (T_n(x)-T_n(y))\big | \rho(x) \rho(y)\di x \di y .
\end{equation*}
For notational convenience, we set $v_n := u_n \circ T_n$ and we drop the tilde in the above expression. Hence, after a change of variables, we consider the following sequence
\begin{align*}
\frac{1}{\varepsilon_n^2} \int_{\frac{D-D}{\varepsilon_n}} \int_{D\cap(D-\varepsilon_n \xi)} \eta(\xi) \big |(v_n (x+\varepsilon_n\xi)-v_n(x)) \cdot (T_n(x+\varepsilon_n\xi)-T_n(x))\big | \rho(x+\varepsilon_n\xi) \rho(x)\di x \di \xi.
\end{align*}

In particular, for any function $w\colon D \to \R^d$ we define the following functional, which only takes into account the integral in the $x$-variable:
%\begin{align*}
% \frac{1}{\varepsilon^2} \int_{\Omega \cap (\Omega-\varepsilon \xi)}  \big |(u \circ T_n (x+\varepsilon\xi)-u\circ T_n(x)) \cdot (T_n(x+\varepsilon\xi)-T_n(x))\big | \rho(x+\varepsilon\xi) \rho(x),
%\end{align*}
%and we set
\begin{align*}
F_n^\xi(w,D):= \frac{1}{\varepsilon_{n}^2} \int_{D \cap (D-\varepsilon_{n} \xi)} \! \big |(w (x+\varepsilon_{n}\xi)-w (x)) \cdot (T_n(x+\varepsilon_{n}\xi)-T_n(x))\big | \rho(x+\varepsilon_{n}\xi) \rho(x)\di x.
\end{align*}
As we are going to work with $1$-dimensional slices of the functions $v_{n}$, for a function $w \colon \R^{d} \to \R^{d}$ we introduce the notation
\[
w^{\xi,y}(t):=w (y+t \xi) \qquad \text{for $\xi \in \mathbb{R}^{d} \setminus\{0\}$, $y \in \Pi^{\xi}$, and $t \in \R$.}
\]
Let $y \in \Pi^\xi$ and $I\subset (D \cap (D-\varepsilon \xi))_y^\xi$, we also define for any $v \colon I \to \R^d$ the functional
\begin{align}
\label{eq:1dfun}
&F^{\xi,y}_{n}(v ,\rho, I):= \frac{1}{\varepsilon_n^2} \int_{I}  \big |(v (t+\varepsilon_n)-v (t )) \cdot (T_n^{\xi, y} ( t + \varepsilon_{n}) -T_n^{\xi, y} (t ))\big | \rho^{\xi, y}( t + \varepsilon_n ) \rho^{\xi, y}(t) \, \de t.
\end{align}

We now prove a one-dimensional lemma similar to \cite[Lemma 3.2]{Gobbino}, where we take the weight~$\rho$ to be identically equal to $1$.

\begin{lemma}\label{lemma:gob3.2}
 Assume that~\eqref{eq:asssumption_epsilon} holds, and let $\xi \in \R^d$ and $y \in \Pi^\xi$. Let $I=[a,b]\subset \R$ be a finite interval and $v_n, v \in L^1(I;\R^d)$  be such that:
    \begin{itemize}
        \item[(i)] $v_n \to v$ in $L^1(I;\R^d)$;
        \item[(ii)] $a$ and $b$ are Lebesgue points of $v$.
    \end{itemize} 
    Then
    \begin{equation}\label{eq3.9g}
    \liminf_{n  \to \infty} F^{\xi,y}_n(v_n,1,I) \ge  |(v (b) - v (a))\cdot \xi|\,.
    \end{equation}
\end{lemma}

\begin{proof}
To simplify the notation, let us assume that $a=0$.  Let $T_n$ be the transport map as in \eqref{e:Tn-scaling}.
We notice that in view of \eqref{eq:asssumption_epsilon} and \eqref{e:second-Tn} we have that there exists $L \in (0, +\infty)$ (independent of~$\xi$ and $y$) such that
\begin{align}
& \lim_{n \to \infty} \frac{\|T_n^{\xi,y}-(y+\cdot \xi)\|_{L^\infty}}{\varepsilon_n} = 0\,,\label{e:T-limit-0}\\
& \sup_n\frac{\|T_n^{\xi,y}(\cdot+\varepsilon_n)-2T_n^{\xi,y}(\cdot)+T_n^{\xi,y}(\cdot-\varepsilon_n)\|_{L^\infty}}{\varepsilon_n^2}\le L\,. \label{e:T-sup-L}
\end{align}
To simplify the notation we drop the dependence on $\xi$ and $y$.

Let us set $J:=|(v(b)-  v(0))\cdot \xi|$. If $J=0$ there is nothing to prove. Hence, we can assume that $J>0$. Moreover, up to a subsequence, we can also assume that
\[
\liminf_{n  \to \infty} F_n(v_n, 1, I)=\lim_{n \to \infty} F_n(v_n, 1 , I)
\]
and
\[
v_n(t) \to v(t) \quad \text{for a.e. } t \in I.
\]

Set  $C:=4(1+|\xi|)+|v(b)|+ |v(0)|$.
Fix $\sigma \in (0,J/C]$, and let $N_{\varepsilon_n}=[|I|/\varepsilon_n]$. We define the set
\begin{align*}
C_n:= 
\bigg \{ t \in [0,\varepsilon_n]: \sum_{k=1}^{N_{\varepsilon_n}} & \ \frac{1}{\varepsilon_n}|v_n(t+k\varepsilon_n)-v_n(t+(k-1)\varepsilon_n)\\
&  \cdot (T_n(t+k\varepsilon_n)-T_n(t+(k-1)\varepsilon_n))| \ge J-  2  C\sigma \bigg\}.
\end{align*}
% \begin{align*}
% &C_n:=\\
% &\left \{ t \in [a,a+\varepsilon_n]: \sum_{k=1}^{N_{\varepsilon_n}}\frac{1}{\varepsilon_n}|v_n(t+k\varepsilon_n)-v_n(t+(k-1)\varepsilon_n) \cdot (T_n(t+k\varepsilon_n)-T_n(t+(k-1)\varepsilon_n))|\ge J-C\eta \right\}.
% \end{align*}
We subdivide the proof in two steps.

\noindent{\em Step 1: } We show that
\begin{equation}\label{eq3.10g}
\lim_{n \to \infty} \frac{|C_n|}{\varepsilon_n} = 1.
\end{equation}
Notice that, by a change of variables,~\eqref{eq3.10g} is equivalent to
\begin{equation}
\label{eq4:10gg}
    \lim_{n\to \infty} |C^{\varepsilon_{n}}_{n}| = 1\,,
\end{equation}
where we have set
\begin{align*}
    C^{\varepsilon_{n}}_{n}:= \bigg\{ \tau \in [0,1]: \sum_{k=1}^{N_{\varepsilon_n}} & \ \frac{1}{\varepsilon_n}|v_n( \varepsilon_{n} \tau+k\varepsilon_n)-v_n( \varepsilon_{n} \tau+(k-1)\varepsilon_n)\\
&  \cdot (T_n(  \varepsilon_{n} \tau +k\varepsilon_n)-T_n(  \varepsilon_{n} \tau +(k-1)\varepsilon_n))| \ge J- 2 C\sigma \bigg\}\,.
\end{align*}

To this end, for all $\delta>0$ we set
\begin{align*}
A_\delta&:=\left \{ t \in [0,\delta]: |v(t)-  v(0)|< \sigma\right \},\\
B_\delta&:=\left \{ t \in [b-\delta,b]: |v(t)-v(b)|< \sigma\right \}.
\end{align*}
By hypothesis {\em (ii)} we have that
\begin{equation}\label{eq3.11g}
\lim_{\delta \to 0} \frac{|A_\delta|}{\delta}= \lim_{\delta \to 0} \frac{|B_\delta|}{\delta}=1.
\end{equation}

By Severini--Egorov Theorem, there exists $I_\delta \subset I$ such that $|I \setminus I_\delta|<\delta^2$, $v_n(t) \to v(t)$ and $T_n(t)\to y+t\xi$ uniformly in $t \in I_\delta$.  We set $c_n(t):=T_n(t)-T_n(t-\varepsilon_n)$ and observe that $c_n/\varepsilon_n \to \xi$ uniformly in $I_\delta$.
Then, there exists $\bar n=\bar n(\sigma,\delta,I_\delta)$ such that for $n\ge \bar n$ it holds
\begin{align}
\label{e:implication}
t' \in B_\delta \cap I_\delta, &  \,t \in A_\delta \cap I_\delta 
 \ \Longrightarrow \ \Big|v_n(t')\cdot \frac{c_n(t')}{\varepsilon_n}-v_n(t) \cdot \frac{c_n(t+\varepsilon_n)}{\varepsilon_n}\Big|\ge J-C\sigma \,.
\end{align}
%where  $C=4(1+|\xi|)+|v(b)|+|v(a)|$.
%where $C=2+2(b-a)|\xi|+|v(b)-v(a)|$. 
Moreover, we have
\begin{equation}\label{eq3.13g}
    |A_\delta \cap I_\delta|\ge |A_\delta|-\delta^2, \quad |B_\delta \cap I_\delta|\ge |B_\delta|-\delta^2.
\end{equation}

% Indeed, it holds
% \begin{align*}
% &\left |(v_n(t)-v_n(t'))\cdot \frac{c_n(t+\varepsilon_n)-c_n(t')}{\varepsilon_n}-(v(b)-v(a))\cdot(t-t')\xi\right | \\
% &\le \left |((v_n(t)-v_n(t'))-(v(t)-v(t')) \cdot \Big(\frac{c_n(t+\varepsilon_n)-c_n(t')}{\varepsilon_n}-(t-t')\xi\Big)\right |\\
% &\quad + \left |((v_n(t)-v_n(t'))-(v(t)-v(t')) \cdot (t-t')\xi\right |\\
% &\quad + \left |((v(t)-v(t'))-(v(b)-v(a)) \cdot ((T_n(t)-T_n(t'))-(t-t')\xi)\right |\\
% &\quad + \left |((v(t)-v(t'))-(v(b)-v(a)) \cdot (t-t')\xi\right |\\
% &\quad + \left |(v(b)-v(a)) \cdot ((T_n(t)-T_n(t'))-(t-t')\xi)\right |\\
% &\le \eta^2+ \eta |t-t'||\xi|+ \eta^2 + \eta |t-t'||\xi|+|v(b)-v(a)| \eta\\
% &\le  (2+2(b-a)|\xi|+|v(b)-v(a)|)\eta=C\eta .
% \end{align*}

Let  $\tau \in [0, 1]$ be such that there exists $M_{n, \tau} \leq N_{\varepsilon_{n}}$ such that $ \varepsilon_{n} \tau \in A_\delta \cap I_\delta$ and $ \varepsilon_{n} \tau +M_{n, \tau}\varepsilon_n \in B_\delta \cap I_\delta$.  
By the triangle inequality and by~\eqref{e:implication} we have
\begin{align}
\label{e:2.8}
        & \liminf_{n \to \infty}  \frac{1}{\varepsilon_n} \sum_{k=1}^{N_{\varepsilon_n}}|(v_n(\varepsilon_{n} \tau +k \varepsilon_n)-v_n(\varepsilon_{n} \tau +(k-1)\varepsilon_n))\cdot (T_n(\varepsilon_{n} \tau  +k\varepsilon_n)-T_n(\varepsilon_{n} \tau +(k-1)\varepsilon_n))| \nonumber
        \\
        & 
        \geq \liminf_{n \to \infty}   \frac{1}{\varepsilon_n} \sum_{k=1}^{M_{n, \tau}} |(v_n(\varepsilon_{n} \tau  +k \varepsilon_n)-v_n(\varepsilon_{n} \tau +(k-1)\varepsilon_n))\cdot (T_n(  \varepsilon_{n} \tau  +k\varepsilon_n)-T_n(  \varepsilon_{n} \tau +(k-1)\varepsilon_n))| \nonumber
        \\
        &
        \geq \liminf_{n \to \infty}  \frac{1}{\varepsilon_n}\bigg |   \sum_{k=1}^{M_{n, \tau} }  (v_n(  \varepsilon_{n} \tau  +k \varepsilon_n)-v_n(   \varepsilon_{n} \tau  +(k-1)\varepsilon_n))\cdot c_n(  \varepsilon_{n} \tau +k \varepsilon_n) \bigg |\nonumber
        \\
        &
        \ge\liminf_{n \to \infty} \frac{1}{\varepsilon_n} \bigg [\Big| v_n(   \varepsilon_{n} \tau  +M_{n, \tau} \varepsilon_n)\cdot c_n(    \varepsilon_{n} \tau  +M_{n, \tau} \varepsilon_n) -v_n(    \varepsilon_{n} \tau  )\cdot c_n(    \varepsilon_{n} \tau  +\varepsilon_n)   \Big |\nonumber 
        \\
        &
        \qquad \qquad \qquad - \bigg | \sum_{k=1}^{M_{n, \tau}-1}  (c_n(   \varepsilon_{n} \tau  +(k+1)\varepsilon_n)-c_n(   \varepsilon_{n} \tau  +k\varepsilon_n))\cdot v_n(  \varepsilon_{n} \tau  +k\varepsilon_n)  \bigg |\bigg] \nonumber\\
        &
        \ge  J - C\sigma  - \limsup_{n \to \infty}  \bigg | \sum_{k=1}^{M_{n, \tau}-1}  \frac{c_n(   \varepsilon_{n} \tau  +(k+1)\varepsilon_n)-c_n(  \varepsilon_{n} \tau  +k\varepsilon_n)}{\varepsilon_n}\cdot v_n(   \varepsilon_{n} \tau  +k\varepsilon_n)  \bigg |. 
\end{align}

We now show that  
\begin{equation}\label{eq:c_nzero}
\begin{split}
     \limsup_{n \to \infty}\bigg | \sum_{k=1}^{M_{n, \tau}-1}   \frac{c_n(   \varepsilon_{n} \tau   +(k+1)\varepsilon_n)-c_n(     \varepsilon_{n} \tau   +k\varepsilon_n)}{\varepsilon_n^2}\cdot v_n(    \varepsilon_{n} \tau   +k\varepsilon_n) \varepsilon_n  \bigg | =0.
\end{split}
\end{equation}
Let us consider the functions   $f_{n,\tau} \colon I \to \R^d$ and $v_{n,\tau} \colon I \to \R^d$   defined by 
\begin{align*}
       f_{n,\tau}   (s) & :=\sum_{k=1}^{N_{\varepsilon_n}-1} \frac{c_n (    \varepsilon_{n} \tau   +(k+1)\varepsilon_n)-c_n(   \varepsilon_{n} \tau   +k\varepsilon_n)}{\varepsilon_n^2} \, {\bf 1}_{[     \varepsilon_{n} \tau   +k\varepsilon_n,     \varepsilon_{n} \tau   +(k+1)\varepsilon_n)}(s)\,,\\
       v_{n,\tau}  (s)& :=\sum_{k=1}^{N_{\varepsilon_n}-1} v_n(   \varepsilon_{n} \tau   +k\varepsilon_n) {\bf 1}_{[   \varepsilon_{n} \tau  +k\varepsilon_n,    \varepsilon_{n} \tau   +(k+1)\varepsilon_n)}(s)
\end{align*} 
for $s \in I $. 
We remark that  by \eqref{e:T-sup-L}   the functions   $f_{n,\tau}$   satisfy a uniform $L^\infty$-bound.
Consider $  0<   c<d<b$ and let $n_{\varepsilon_n}^c:=[  c   /\varepsilon_n]$, $N_{\varepsilon_n}^d:=[  d   /\varepsilon_n]$. We first show that
\begin{equation}
\label{eq:intf_ntzero}
\lim_{n \to \infty} \Big | \int_c^d   f_{n,\tau}   (s) \,\de s\Big |=0 .
\end{equation}
Indeed, for $w \in \R^d$ it holds
\begin{align*}
    0 &= \limsup_{n \to \infty} \Big | \frac{1}{\varepsilon_n} \sum_{k=n_{\varepsilon_n}^c+1}^{N_{\varepsilon_n}^d} (w-w) \cdot (T_n(    \varepsilon_{n} \tau   +k\varepsilon_n)-T_n(    \varepsilon_{n} \tau   +(k-1)\varepsilon_n)) \Big |\\
    &  \geq   \limsup_{n \to \infty} \frac{1}{\varepsilon_n} \bigg [-\Big| w\cdot c_n(    \varepsilon_{n} \tau   +N_{\varepsilon_n}^d \varepsilon_n) -w\cdot c_n(    \varepsilon_{n} \tau   + n_{\varepsilon_n}^c   \varepsilon_n)   \Big | \\
    &\qquad \qquad+ \Big |\sum_{k=n_{\varepsilon_n}^c+1}^{N_{\varepsilon_n}^d-1} (c_n(     \varepsilon_{n} \tau   +(k+1)\varepsilon_n)-c_n(    \varepsilon_{n} \tau   +k\varepsilon_n))\cdot w  \Big |\bigg]\\
    &= \limsup_{n \to \infty} \Big |\sum_{k=n_{\varepsilon_n}^c+1}^{N_{\varepsilon_n}^d-1} \frac{(c_n(     \varepsilon_{n} \tau   +(k+1)\varepsilon_n)-c_n(   \varepsilon_{n} \tau   +k\varepsilon_n))}{\varepsilon_n^2}\cdot w  \varepsilon_n \Big | \\
    &= \limsup_{n \to \infty} \Big |   \int_{ (n^{c}_{\varepsilon_{n}} + 1 + \tau)\varepsilon_n}^{(N^{d}_{\varepsilon_{n}} + \tau) \varepsilon_{n}}   w \cdot   f_{n,\tau} (s)\, \de s\Big | .
\end{align*}
This proves \eqref{eq:intf_ntzero}.
Combining \eqref{eq:intf_ntzero} together with \eqref{e:T-limit-0} we deduce   $f_{n,\tau}\rightharpoonup 0$   weakly$^*$ in $L^\infty(I; \R^{d})$.
This in turn implies \eqref{eq:c_nzero} after showing   $v_{n,\tau} \to v$   strongly in $L^1(I; \R^{d})$ for a.e.   $\tau \in [0, 1]$.   To this end, we consider
% \begin{align*}
%     &\fint_a^{a+\varepsilon} \Big (\int_a^b |v_{n,t}(s)-v(s)| \,\de s \Big) \de t \\
%     &= \fint_a^{a+\varepsilon} \Big (\int_a^b\Big | \sum_{k}  1_{[t+k\varepsilon_n, t+(k+1)\varepsilon_n)}(s) (v_n(t+k\varepsilon_n)- v(s)) \Big|\,\de s \Big ) \de t\\
%     &= \sum_k\fint_a^{a+\varepsilon} \Big (\int_{t+k\varepsilon_n}^{t+(k+1)\varepsilon_n} |  (v_n(t+k\varepsilon_n)- v(s)) |\,\de s \Big ) \de t\\
%     &=\sum_k\fint_a^{a+\varepsilon} \Big (\int_{t}^{t+\varepsilon_n} |  (v_n(t+k\varepsilon_n)- v(s+k\varepsilon_n)) |\,\de s \Big ) \de t\\
%     &\le \sum_k \fint_a^{a+\varepsilon_n} \Big (\int_t^{t+\varepsilon_n} |v_n(t+k \varepsilon_n)-v(t+k\varepsilon_n)|\, \de s\Big )\de t\\
%     &\quad + \sum_k \fint_a^{a+\varepsilon_n} \Big (\int_t^{t+\varepsilon_n} |v(t+k \varepsilon_n)-v(s+k\varepsilon_n)|\,\de s \Big )\de t\\
%     &\le \int_a^b |v_n(t)-v(t)| \,\de t  + \varepsilon_n \sum_k \fint_a^{a+\varepsilon_n} \Big (\fint_t^{t+\varepsilon_n} |v(t+k \varepsilon_n)-v(s+k\varepsilon_n)|\,\de s \Big )\de t
% \end{align*}
\begin{align*}
      \int_0^{1}   & \Big (\int_0^b |  v_{n,\tau}  (s)-v(s)| \,\de s \Big) \de  \tau   \\
    &=   \int_0^{1}   \Big (\int_0^b\Big | \sum_{k=1}^{N_{\varepsilon_n}-1}  {\bf 1}_{[   \varepsilon_{n} \tau   +k\varepsilon_n,   \varepsilon_{n} \tau   +(k+1)\varepsilon_n)}(s) (v_n(  \varepsilon_{n} \tau   +k\varepsilon_n)- v(s)) \Big|\,\de s \Big ) \de   \tau   \\
    &= \sum_{k=1}^{N_{\varepsilon_n}-1}   \int_0^{1}  \Big (\int_{  \varepsilon_{n} \tau  +k\varepsilon_n}^{  \varepsilon_{n} \tau  +(k+1)\varepsilon_n}   |  v_n(  \varepsilon_{n} \tau   +k\varepsilon_n)- v(s) |\,\de s \Big ) \de   \tau   \\
    &=\sum_{k=1}^{N_{\varepsilon_n}-1} \frac{1}{\varepsilon_n} \int_{k\varepsilon_n}^{(k+1)\varepsilon_n} \Big (\int_{0}^{\varepsilon_n} |  v_n(t)- v(t+h) |\,\de h \Big ) \de t\\
    &\le \int_0^b \Big (\fint_0^{\varepsilon_n} |  v_n(t)- v(t+h) |\,\de h \Big ) \de t\\
    & \le \int_0^b |  v_n(t)- v(t) |\, \de t +  \int_0^b \Big (\fint_0^{\varepsilon_n} |  v(t)- v(t+h) |\,\de h  \Big ) \de t\,.
\end{align*}
The first term on the right-hand side tends to zero since $v_n \to v$ in $L^1(I;\R^d)$. The second term tends to zero by  standard convolution estimates, since $v \in L^{1} (I; \R^{d})$. 
%the dominated convergence theorem since $\fint_0^{\varepsilon_n} |v(t)-v(t+h)|\, \de h \le g_n(t):=|v(t)|+ \fint_0^{\varepsilon_n} |v(t+h)| \, \de h \to 2|v(t)|$ for a.e. $t \in I$ and $v \in L^1(I;\R^d)$.
Hence, 
\begin{align*}
    \lim_{n \to \infty}   \int_0^{1}   \Big (\int_0^b |   v_{n,\tau}   (s) - v(s) | \,\de s \Big) \de   \tau   =  0\,,
\end{align*}
which implies the claim. Therefore, \eqref{eq:c_nzero} is proved. Combining~\eqref{e:2.8} and~\eqref{eq:c_nzero} we infer that   for $\tau \in [0, 1]$ such that there exists $M_{n, \tau} \leq N_{\varepsilon_{n}}$ with  $ \varepsilon_{n} \tau + M_{n, \tau} \varepsilon_n \in B_\delta \cap I_\delta$ and $ \varepsilon_{n} \tau \in A_\delta \cap I_\delta$  
\begin{align}
\label{eq:scdis}
        \liminf_{n \to \infty} \frac{1}{\varepsilon_n} & \sum_{k=1}^{M_{n, \tau}}|(v_n(  \varepsilon_{n} \tau   +k \varepsilon_n)-v_n(    \varepsilon_{n} \tau   +(k-1)\varepsilon_n))\cdot (T_n(    \varepsilon_{n} \tau    +k\varepsilon_n)-T_n(    \varepsilon_{n} \tau    +(k-1)\varepsilon_n))| \nonumber
        \\
        & \vphantom{ \sum_{k=1}^{N_{\varepsilon_n}}} \ge  J - C \sigma \,.
\end{align}

We now proceed as in \cite[Lemma 3.2]{Gobbino}. 
Let us set
\begin{align*}
    K_n^0&:= \big\{k \in \N: [k\varepsilon_n, (k+1)\varepsilon_n]\subset [0,\delta] \big\},\\
    K_n^b&:= \big\{k \in \N: [k\varepsilon_n, (k+1)\varepsilon_n]\subset [b-\delta,b] \big\}.
\end{align*}
It is easy to see that
\begin{equation}\label{eq3.14g}
    \# K_n^0=[\delta/\varepsilon_n], \quad \#K_n^b\ge ([\delta/\varepsilon_n]-1).
\end{equation}
Let us further set
\begin{align}
      A_\delta^{n, k}   &:= \big\{  \tau \in [0, 1] :\,  \varepsilon_{n} \tau   +k\varepsilon_n\notin A_\delta\cap I_\delta\big\} \qquad \text{  for $ k \in K_n^0$  } ,\label{eq3.15g}\\
      B_\delta^{n, k}  &:= \big\{  \tau \in [0, 1]: \,   \varepsilon_{n} \tau   +k\varepsilon_n\notin B_\delta\cap I_\delta \big\} \qquad \text{  for $ k \in K_n^b$  }.\label{eq3.16g}
\end{align}
By \eqref{eq3.15g}, \eqref{eq3.16g} and \eqref{eq:scdis} it follows that
  
\begin{align}\label{eq3.17g}
    \tau \in [0, 1] \setminus & \bigg(\bigcup_{k \in K_{n}^{0}} A_\delta^{n,k}\cup \bigcup_{k \in K_{n}^{b}} B_\delta^{n,k} \bigg) 
    \\
    &
    \nonumber \vphantom{\bigcup_{k \in K}} \Longrightarrow \, \tau \in C^{\varepsilon_{n}}_n \text{ for $n$ sufficiently large (depending on~$\tau$ and~$\delta$).}
\end{align}  
  By \eqref{eq3.13g} we have that
%\begin{align*}
 %   &\bigcup_{k \in K_n^a} \big( a + \varepsilon_{n} A_\delta^{n, k}+k\varepsilon_n \big)  \subset [a,a+\delta]\setminus (A_\delta \cap I_\delta),\\
  %  &\bigcup_{k \in K_n^b} \big( a + \varepsilon_{n} B_\delta^n+k\varepsilon_n \big)\subset [b-\delta,b]\setminus (B_\delta \cap I_\delta),
%\end{align*}
%and the above unions are disjoint, by  and \eqref{eq3.14g} it follows that
\begin{align*}
    | A_\delta^{n, k}|&\le \frac{1}{\varepsilon_{n}}(\delta-|A_\delta|+\delta^2) \qquad \text{for $k \in K_{n}^{0}$},\\
    |B_\delta^{n, k}|&\le \frac{1}{\varepsilon_{n}}(\delta-|B_\delta|+\delta^2) \qquad \text{for $k \in K_{n}^{b}$}.
\end{align*}
Hence, by \eqref{eq3.17g}
\begin{align*}
\liminf_{n \to \infty} |C^{\varepsilon_{n}}_n| & \ge  \liminf_{n \to \infty} \bigg| [0, 1] \setminus  \bigg(\bigcup_{k \in K_{n}^{0}} A_\delta^n\cup \bigcup_{k \in K_{n}^{b}} B_\delta^n \bigg) \bigg| - \bigg|\bigg[ [0, 1] \setminus  \bigg(\bigcup_{k \in K_{n}^{0}} A_\delta^n\cup \bigcup_{k \in K_{n}^{b}} B_\delta^n \bigg)\bigg] \setminus C_{n}^{\varepsilon_{n}} \bigg| 
\\
&
\geq 1 - \sum_{k \in K^{0}_{n}} |A^{n, k}_{\delta}| - \sum_{k \in K^{b}_{n}} |B^{n, k}_{\delta}|
\\
&
\geq 1 - \frac{1}{\varepsilon_n}(\delta-|A_\delta|+\delta^2) \Big [\frac{\delta}{\varepsilon_n} \Big ]^{-1}- \frac{1}{\varepsilon_{n}}(\delta-|B_\delta|+\delta^2)\left (\Big [\frac{\delta}{\varepsilon_n} \Big ]-1\right )^{-1} 
\\
&
= 1 - \bigg( 1 - \frac{|A_{\delta}|}{\delta} + \delta \bigg) - \bigg( 1 - \frac{|B_{\delta}|}{\delta} + \delta \bigg)
\\
&
= \frac{|A_{\delta}|}{\delta}  + \frac{|B_{\delta}|}{\delta}  -1-2\delta\,.
\end{align*} 
% Dividing by $\varepsilon_n$ and taking the liminf as $n \to \infty$, we get
%\[
%\liminf_{n \to \infty}\frac{|C_n|}{\varepsilon_n}\ge 1-\frac{1}{\delta}(2\delta+2\delta^2-|A_\delta|-|B_\delta|).
%\]
Since $\delta>0$ is arbitrary, we can send $\delta \to 0^+$ and by \eqref{eq3.11g} we obtain
\begin{align*}
    \liminf_{n\to\infty} |C^{\varepsilon_{n}}_{n}| \geq 1\,.
\end{align*}
Since $|C^{\varepsilon_{n}}_n|\le 1$, it must be
\[
\lim_{n\to \infty} \, |C^{\varepsilon_{n}}_n| = 1\,.
\]
Thus, equalities \eqref{eq4:10gg} and \eqref{eq3.10g} are proved.

\noindent{\em Step 2: } Let us prove \eqref{eq3.9g}. By definition of $C_n$ we have
\begin{align*}
    &F_n(v_n, 1 , I)\ge F_n(v_n, 1, [0,\varepsilon_nN_{\varepsilon_n}])\\
    &=\frac{1}{\varepsilon_n^2}\int_0^{\varepsilon_nN_{\varepsilon_n}} \big |(v_n (t+\varepsilon_n)-v_n(t) \cdot (T_n(t+\varepsilon_n)-T_n(t))\big |  \, \de t\\
    &=\frac{1}{\varepsilon_n^2}\int_0^{\varepsilon_n} \sum_{k=1}^{N_{\varepsilon_n}}\big |(v_n (t+k\varepsilon_n)-v_n(t+(k-1)\varepsilon_n) \cdot (T_n(t+k\varepsilon_n)-T_n(t+(k-1)\varepsilon_n))\big | \, \de t \\
    &\ge\frac{|C_n|}{\varepsilon_n} (J- \textcolor{blue}{2}C\sigma).
\end{align*}
By sending $ n \to \infty$ and then $\sigma \to 0$,  we deduce
\[
\liminf_{n \to \infty} F_n(v_n, 1 , I) \ge J.
\]
This concludes the proof of~\eqref{eq3.9g} and of the lemma.
\end{proof}

For the proof of the liminf we need the following technical lemma which is a generalization of \cite[Lemma 3.3]{Gobbino}.

\begin{lemma}\label{lemma:3.3gob}
    Let $u \in L^1_{loc}(\R)$. Then there exists $a \in \R$ such that
    \begin{itemize}
        \item[(i)] $a+q$ is a Lebesgue point of $u$ for every $q \in \Q$;
        \item[(ii)] every sequence $\{u_n\}_{n \in\N}\subset L^1_{loc}(\R)$ that satisfies the conditions:
        \begin{itemize}
            \item $u_n\big(a+\frac{z}{n}\big)=u\big(a+\frac{z}{n} \big)$ for all $z \in \Z$,
            \item if $x \in \big [a+\frac{z}{n},a+\frac{z+1}{n}\big]$, then $u_n(x)$ belongs to the interval with endpoints $u\big ( a+\frac{z}{n}\big)$ and $u\big ( a+\frac{z+1}{n}\big)$,
        \end{itemize}
        has a subsequence converging to $u$ in $L^1_{loc}(\R)$.
    \end{itemize}
\end{lemma}

\begin{proof}
    We only need to show the validity of \textit{(ii)}, since \textit{(i)} is trivially true.

    We use the same notation of \cite[Lemma 3.3]{Gobbino}. For $n \ge 1$, $z \in \Z$, and $a \in [0,1]$, we define for every $x \in \R$
    \[
    v_n^a(x)=u \Big (a+\frac{[n(x-a)]}{n} \Big ).
    \]

    For every fixed interval $I \subset \R$, up to a subsequence, we have
    \[
    v_n^a \to u \quad \text{in } L^1(I) \text{ for a.e. } a \in [0,1].
    \]
    Indeed, it holds
    \[
    \lim_{n \to \infty} \int_I\int_0^1|v_n^a(x)-u(x)| \, \de a \, \de x = \lim_{n \to \infty}  \int_I\fint_{-\frac{1}{n}}^0 |u(x+y)-u(x)| \, \de y \, \de x = 0,
    \]
    by the dominated convergence theorem since $u \in L^1(\R)$.

    The proof then follows in the same way as in \cite[Lemma 3.3]{Gobbino}. 
\end{proof}

We are now in a position to prove Theorem~\ref{thm:liminf}.

\begin{proof}[Proof of Theorem \ref{thm:liminf}]
    Without loss of generality we can assume 
    \begin{equation}
    \label{e:closure0}
    \liminf_{n \to \infty} GTV_{n}(u_n)<\infty.
    \end{equation}
    We define $v_n:=u_n \circ T_n$. Since $(u_n, \nu_{n}) \to (u, \nu)$ in $TL^1$, it holds $v_n \to u$ in $L^1(D;\R^d)$.
    
     Let us recall that, in view of Remark \ref{rmk:1}, we can reduce to the case of $\eta$ of the form $\eta(t)=c$ for $t<b$ and $\eta(t) =0$ for $t\ge b$. Furthermore, we  only have to prove the lower semicontinuity of the one-dimensional functional \eqref{eq:1dfun}. Indeed, 
    we have
    \begin{align}
    \label{e:2000}
        \liminf_{n \to \infty} \frac{1}{\varepsilon_n^2} &  \iint_{D \times D} \eta_{\varepsilon_n}(x-y)|(v_n(x)-v_n(y))\cdot (T_n(x)-T_n(y))|\rho(x)\rho(y) \, \de x \de y
        \\
        &
        =\liminf_{n \to \infty} \int_{\frac{D-D}{\varepsilon_n}} F_n^\xi(v_n,D)\eta(\xi) \,\de \xi \nonumber
        \\
        &
        \ge \int_{\R^d}  \liminf_{n \to \infty} \Big (\int_{\Pi^\xi} F_n^{\xi,y}(   v^{\xi, y}_n   ,\rho, (D\cap (D-\varepsilon \xi))_y^\xi)\, \de \mathcal{H}^{d-1}(y) \Big )|\xi| \eta(\xi) \, \de\xi \nonumber
        \\
        &
        \ge \int_{\R^d}  \int_{\Pi^\xi}  \liminf_{n \to \infty} F_n^{\xi,y}(   v^{\xi, y}_n    ,\rho, (D\cap (D-\varepsilon \xi))_y^\xi) |\xi| \eta(\xi) \, \de \mathcal{H}^{d-1}(y)  \de\xi. \nonumber
    \end{align}

Recall that, for any function $v\colon D \to \R^d$, we have set, for $\xi \in \R^{d}$ and $y \in \Pi^\xi$, $D^\xi:=\{t \in \R: y+t\xi \in D\}$ and $v^{\xi,y}(t):=v(y+t\xi)$ for $t \in D_{y}^{\xi}$. Since $v_{n} \to v$ in $L^{1} (D; \R^{d})$, we have that $v_{n}^{\xi, y} \to v^{\xi, y}$ in $L^{1} (  D^{\xi}_{y}   ; \R^{d})$ for a.e.~$\xi \in \R^{d}$ and $\mathcal{H}^{d-1}$-a.e.~$y \in \Pi^{\xi}$. By arguing as in  \cite{Gobbino, AlmiDavoliKubinTasso}, we consider $j \in \N$, $a \in \R$ given by Lemma \ref{lemma:3.3gob} applied to the limit function $   \hat{u}^{\xi}_{y} :=   u^{\xi,y} \cdot \xi$, and $I^z_j=\Big [a+\frac{z}{j}, a + \frac{z+1}{j} \Big]$ for every $z \in \Z$ such that $I_j^z \subset D_y^\xi$. We also define 
\[
\rho^{\xi,y}[I_j^z]:=\min_{t \in I_j^z} \rho^{\xi,y}(t)\,.
\]
  For every $K \Subset D^{\xi}_{y}$ let   and $w^{\xi,y}_j  \colon   K   \to \R^d$ to be the piecewise affine function interpolating between $u^{\xi,y} (a+\frac{z}{j}  )$ and $u^{\xi,y} (a+\frac{z}{j}  )$ in $I_j^z$.
By Lemma \ref{lemma:3.3gob} we have, up to a subsequence, $  w^{\xi,y}_j  \cdot \xi \to    \hat{u}^{\xi}_{y}   $ in   $L^1(K).$ Let us further denote by $Z_{j, K}:=\{ z \in \mathbb{Z}: I^{z}_{j} \cap K \neq \emptyset\}$. Notice that for $n$ sufficiently large, $\bigcup_{z \in Z_{j, K}} I^{z}_{j} \subseteq (D \cap (D - \varepsilon_{n}\xi))^{\xi}_{y}$.

We have the following estimates
    \begin{align}
    \label{eq:liminf1d}
\liminf_{n \to \infty} \, & F^{\xi,y}_{n}(   v^{\xi, y}_n   ,\rho, (D \cap (D -   \varepsilon_{n}   \xi))_y^\xi)
% &=\liminf_{n \to \infty}\frac{1}{\varepsilon_n^2} \int_{I}  \big |(v_n^{\xi,y} (t+\varepsilon_n)-v_n^{\xi,y}(t)) \cdot (T_n^{\xi,y}(t+\varepsilon_n)-T_n^{\xi,y}(t))\big | \rho^{\xi,y}(t+\varepsilon_n) \rho^{\xi,y}(t) \, \de t\\
\ge \liminf_{n \to \infty}   \sum_{z \in Z_{j, K}}   F^{\xi,y}_{n}(v_n,\rho, I_j^z)\\
&\ge   \sum_{z \in Z_{j, K}}   \liminf_{n \to \infty}  F^{\xi,y}_{n}(   v^{\xi, y}_n   ,\rho, I_j^z) \ge   \sum_{z \in Z_{j, K}}    \big(\rho^{\xi,y}[I_j^z]\big)^2 \liminf_{n \to \infty}  F^{\xi,y}_{n}(   v^{\xi, y}_n   ,1, I_j^z) \nonumber
\\
&
\hspace{-0.56cm}\stackrel{\textnormal{Lemma \ref{lemma:gob3.2}}}{\ge}   \sum_{z \in Z_{j, K}}    \big( \rho^{\xi,y}[I_j^z]\big)^2 \Big|\Big[u^{\xi,y}\Big ( a+\frac{z+1}{j}\Big)- u^{\xi,y}\Big ( a+\frac{z}{j}\Big) \Big] \cdot \xi \Big|\nonumber
\\
&
  \geq \int_{ K}    |D w^{\xi,y}_j  (t) \cdot \xi| \sum_{z}  \big(\rho^{\xi,y}[I_j^z]\big)^2 \, {\bf 1}_{I_j^z}(t) \, \de t. \nonumber
\end{align}
Since, by assumption $(\rho1)$, $\rho$ is a continuous function  we infer
\begin{align*}
       \sum_{z \in Z_{j, K}}    \big(\rho^{\xi,y}[I_j^z]\big)^2 \, {\bf 1}_{I_j^z} \to (\rho^{\xi,y})^2 \quad \text{uniformly in $K$ as $j \to \infty$}.
\end{align*}
Therefore, using also $(\rho2)$, by the dominated convergence theorem we get for $j \to \infty$
\begin{align}
\label{e:closure2}     
    \liminf_{j \to \infty} &   \int_{K}   |D  w^{\xi,y}_j  (t) \cdot \xi|   \sum_{z \in Z_{j, K}}    \big(\rho^{\xi,y}[I_j^z]\big)^2 \,{\bf 1}_{I_j^z}(t) \, \de t 
    \\
    &
    \geq \liminf_{j \to \infty}   \int_{K}    |D  w^{\xi,y}_j  (t)\cdot \xi| (\rho^{\xi,y}(t))^2  \, \de t \nonumber
    \\
    &
    - \limsup_{j \to \infty}   \int_{K }   |D  w^{\xi,y}_j   (t) \cdot \xi| \,  \bigg|   \sum_{z \in Z_{j, K}}     \big(\rho^{\xi,y}[I_j^z]\big)^2 \, {\bf 1}_{I_j^z}(t) -(\rho^{\xi,y}(t))^2\bigg|  \, \de t \nonumber
     \\ 
    % &=\liminf_{j \to \infty}\int_{ \Omega_y^{\xi/|\xi|}}  |D\bar u^{\xi/|\xi|,y}_j(s) \cdot \xi/|\xi||  \rho^{\xi/|\xi|,y}(s)^2\, \de s\\
    &
     = \liminf_{j \to \infty}   \int_{K}   |D  w^{\xi,y}_j   (t)\cdot \xi| (\rho^{\xi,y}(t))^2  \, \de t \nonumber
    \\
    &
     =  \liminf_{j \to \infty}   \int_{K_{|\xi|}}   \bigg| D  w^{ \xi/|\xi | , y }_j   (t)\cdot \frac{\xi}{|\xi|} \bigg| (\rho^{\xi/|\xi | , y }(t))^2 |\xi|\, \de t \,, \nonumber
    \end{align}
      where we have set $K_{|\xi|}:=\{ t \in \R: t/|\xi| \in K\} \Subset D^{\xi/|\xi|}_{y}$.   Furthermore, $(\rho2)$ also implies that
\begin{align}
\label{e:closure}    
     \liminf_{j \to \infty} &  \int_{K}  |D  w^{\xi,y}_j  (t) \cdot \xi|  \sum_{z \in Z_{j, K} }   \big(\rho^{\xi,y}[I_j^z]\big)^2 \,{\bf 1}_{I_j^z}(t) \, \de t
     \\
     &
     \nonumber \geq    \liminf_{j \to \infty}  \int_{K_{|\xi|}}  \alpha^{2} \bigg| D  w^{ \xi/|\xi | , y }_j   (t)\cdot \frac{\xi}{|\xi|} \bigg| \,|\xi| \, \de t  \,. 
\end{align} 
  By standard lower semicontinuity of the total variation and by the convergence $w^{ \xi/|\xi | , y }_j  \cdot \frac{\xi}{|\xi|} \to \hat{u}^{\xi/|\xi|}_{y}$ in $L^{1} ( K_{|\xi|}  )$,  from~\eqref{e:closure}, we deduce that for a.e.~$\xi \in \R^{d}\setminus\{0\}$ and $\mathcal{H}^{d-1}$-a.e.~$y \in \Pi^{\xi}$ we have
\begin{align}
\label{e:closure3}
  \liminf_{j \to \infty} &  \int_{K}  |D  w^{\xi,y}_j  (t) \cdot \xi|  \sum_{z\in Z_{j, K}}   \big(\rho^{\xi,y}[I_j^z]\big)^2 \,{\bf 1}_{I_j^z}(t) \, \de t \geq  \alpha^{2} |\xi|  \big| D \hat{u}^{ \xi/|\xi | , y }    \big| ( K_{|\xi|} ) \,.
\end{align}
The combination of~\eqref{e:closure3} with~\eqref{e:closure0} and~\eqref{eq:liminf1d} yields
\begin{align*}
    \liminf_{n \to \infty} F^{\xi,y}_n(v^{\xi, y}_n,\rho,  (D \cap (D -  \varepsilon_{n} \xi))_y^\xi  ) \ge  \alpha^{2} |\xi|  \big| D \hat{u}^{ \xi/|\xi | , y }    \big| ( K_{|\xi|} )
\end{align*}
for every $K \Subset D^{\xi}_{y}$. Taking the limit as $K \nearrow D^{\xi}_{y}$ we have that $K_{|\xi|} \nearrow D^{\xi/|\xi|}_{y}$ and
\begin{align*}
    \liminf_{n \to \infty} F^{\xi,y}_n(v^{\xi, y}_n,\rho,  (D \cap (D -  \varepsilon_{n} \xi))_y^\xi  ) \ge  \alpha^{2} |\xi|  \big| D \hat{u}^{ \xi/|\xi | , y }    \big| ( D^{\xi/|\xi|}_{y} )\,.
\end{align*}
Thus, we infer that $u \in BD(D)$ (cf.~\cite{Ambrosio-Coscia-DalMaso}).  In a similar way,~\eqref{eq:liminf1d}--\eqref{e:closure2}  imply that
\begin{equation}
\label{e:2001}
    \liminf_{n \to \infty} F^{\xi,y}_n(v^{\xi, y}_n,\rho, (D \cap (D -  \varepsilon_{n} \xi))_y^\xi) \ge \int_{D_y^{\xi/|\xi|}} \rho^{\xi/|\xi| ,y}(t)^2 \,|\xi| \,  \de | \hat{u}^{\xi/|\xi|,y}| (t)\,.
\end{equation}
From \eqref{e:2000} and~\eqref{e:2001} we get that
    \begin{align}
    \label{e:almost-done}
        \liminf_{n \to \infty} \frac{1}{\varepsilon_n^2} & \iint_{D \times D} \eta_{\varepsilon_n}(x-y)|(v_n(x)-v_n(y))\cdot (T_n(x)-T_n(y))|\rho(x)\rho(y) \, \de x \de y
        \\
        &
        \ge \int_{\R^d} \int_{\Pi^\xi} \bigg(  \int_{D_y^{\xi/|\xi|}} \rho^{\xi/|\xi| ,y}(t)^2 \,  \de | \hat{u}^{\xi/|\xi|,y}| (t) \bigg) |\xi|^{2} \eta(\xi) \, \de \mathcal{H}^{d-1} (y) \, \de \xi \nonumber
        \\
        &
        = \int_{\R^d}\bigg(\int_{D} \rho(x)^2 \,\de\Big|Eu(x) \, \frac{\xi}{|\xi|}\cdot \frac{\xi}{|\xi|}\Big| \bigg) |\xi|^2\eta(\xi) \,  \de \xi \nonumber
        \\
        &
        =\int_{ D  }\bigg(\int_{\R^d}  \eta(\xi)  \Big|\frac{Eu(x)}{|Eu(x)|} \, \xi \cdot \xi\Big| \,\de \xi\bigg)  \rho(x)^2\,  \de |Eu(x)|. \nonumber
    \end{align} 
    Recalling the norm on symmetric  matrices $\phi_{\eta} \colon \mathbb{M}^{d}_{sym} \to [0,\infty)$ defined as
    \[
    \phi_{\eta}(A)= \int_{\R^d} \eta(\xi) |A\xi \cdot \xi|\, \de \xi\,,
    \]
    we rewrite~\eqref{e:almost-done} as
    \begin{align*}
        \liminf_{n \to \infty} \frac{1}{\varepsilon_n^2} & \iint_{D \times D} \eta_{\varepsilon_n}(x-y)|(v_n(x)-v_n(y))\cdot (T_n(x)-T_n(y))|\rho(x)\rho(y)  \de x \de y\\
        &=\int_{D }\rho(x)^2\phi_{\eta}\bigg ( \frac{Eu(x)}{|Eu(x)|}\bigg ) \,  \de |Eu(x)|=TV_{\eta}(u;\rho^2)\,.
    \end{align*}
    This concludes the proof of the liminf inequality.
\end{proof}

\section{Construction of a recovery sequence}

In this section we establish the limsup inequality.

\begin{theorem}\label{thm:limsup}
    Assume $({\rm K1})$--$({\rm K3})$, $(\rho1)$--$(\rho2)$, and~\eqref{eq:asssumption_epsilon}. For every 
    $u \in L^1(D;\R^d;\nu)$ there exists $\{u_n\}_{n \in \N} \subset L^1(\{X_1,\ldots, X_m\};\R^d;\nu_n)$ such that $(\nu_n,u_n) \to (\nu,u)$ in $TL^1$ and
    \begin{equation}
    \label{e:limsup-ineq}
        \limsup_{n \to \infty} GTV_{n}(u_{\varepsilon_n}) \le TV_\eta(u;\rho^2) \,.
    \end{equation}
\end{theorem}

\begin{proof}[Proof of Theorem \ref{thm:limsup}]
    Without loss of generality we can assume that 
    \[
    TV_\eta(u;\rho^2) < \infty\,.
    \]
    Hence, by $(\rho2)$ we have that $u \in BD(D)$. Thus, we can approximate $u$ with a sequence $u_k \in C^\infty(D;\R^d) \cap \textnormal{Lip}(D;\R^d)$ such that $u_k \to u$ in $L^1(D;\R^d)$, $Eu_k \stackrel{*}{\rightharpoonup} Eu$, and $|Eu_k|(D)\to |Eu|(D)$. By Reshetnyak continuity theorem (see, e.g.,~\cite[Theorem~2.39]{AFP}) we have $TV_\eta(u_k;\rho^2) \to TV_\eta(u;\rho^2)$. Therefore it is enough to prove the limsup inequality for $u_k$. To simplify the notation we drop the dependence on $k$.

    % By Remark \ref{rmk:etalimsup} we can further assume that $\eta$ is of the form $\eta(t)=\alpha$ for $t <b$ and $\eta(t)=0$ for $t\ge b$. 

     Arguing  as in \cite[Section~5]{Trillos-Slepcev}, we may assume that the kernel $\eta$ is of the form $\eta(t)=c$ for $t<b$ and $\eta(t)=0$ for $t\ge b$. Indeed, if we can establish the limsup inequality under this assumption, then, by the subadditivity of the limsup, the same inequality also holds for functions of the form $\eta=\sum_{k=1}^l \eta_k$, for some $l\in\mathbb{N}$, where each $\eta_k$ is of the above form. Then, we can extend to the case $\eta$ compactly supported by approximation by a sequence of piecewise constant functions $\eta_n\colon [0,+\infty)\to [0,+\infty)$ such that $\eta_n \searrow \eta$ almost everywhere. Finally, to prove the limsup inequality in the general case of $\eta$ with possibly unbounded support, we consider $\eta_\alpha \colon [0,+\infty)\to [0,+\infty)$ defined by $\eta_\alpha(t):=\eta(t)$ for $t \le \alpha$ and $\eta_\alpha(t):=0$ for $t>\alpha$. Then the energy can be rewritten as
    \begin{align*}
        & GTV_{n}(u)= \ GTV_n^\alpha(u)\\
        & \qquad + \frac{1}{\varepsilon_n^2} \int_{{}_{\scriptstyle \{|T_n(x)-T_n(y)|>\alpha \varepsilon_n \}} } \!\!\!\!\!\!\!\!\!\!\!\!\!\!\!\!\!\!\!\!\!\!\!\!\!\!\!\!\!\!\!\!\!\!\!\!\!\!\ \eta_{\varepsilon_n}(T_n(x)-T_n(y))  |(u \circ T_n(x)-u \circ T_n(y))\cdot (T_n(x)-T_n(y))| \rho(x) \rho (y) \,\de x \de y,
    \end{align*}
    where $GTV^\alpha_n$ denotes the energy with $\eta_\alpha$ in place of $\eta$. Since it is enough to prove the limsup estimate for Lipschitz functions, proceeding as in \cite[Proof of Theorem~1.1, Step~4]{Trillos-Slepcev} one can show that the error $GTV_{n}(u)-GTV_n^\alpha(u)$ tends to zero. Hence, it suffices to reduce to the case in which $\eta$ is compactly supported.

    Set $\tilde \varepsilon_n:=\varepsilon_n-   \frac{2}{b}\|Id-T_n\|_{L^\infty}$. By the assumption (K2) we infer
    \[
    \eta \Big ( \frac{|T_n(x)-T_n(y)|}{\varepsilon_n}\Big ) \le \eta \Big ( \frac{|x-y|}{\tilde \varepsilon_n} \Big ),
    \]
    for almost every $(x,y) \in D \times D$.
    Thus we have the upper bound
    \begin{align*}
        &GTV_n(u) \\
        &\le \left (\frac{\tilde \varepsilon_n}{\varepsilon_n} \right)^{d+2} \frac{1}{\tilde \varepsilon_n^2}\iint_{D \times D} \frac{1}{\tilde\varepsilon_n^d}\eta \left (\frac{|x-y|}{\tilde \varepsilon_n} \right ) |(u\circ T_n(x)-u\circ T_n(y))\cdot (T_n(x)-T_n(y))|\rho(x) \rho(y) \di x \di y \\
        &=\left (\frac{\tilde \varepsilon_n}{\varepsilon_n} \right)^{d+2} \frac{1}{\tilde \varepsilon_n^2}\iint_{D \times D} \eta_{\tilde \varepsilon_n} (x-y  ) |(u\circ T_n(x)-u\circ T_n(y))\cdot (T_n(x)-T_n(y))|\rho(x) \rho(y)\di x \di y .
    \end{align*}
    Since $\frac{\tilde \varepsilon_n}{\varepsilon_n} \to 1$, it suffices to prove the limsup inequality for
    \[
    \frac{1}{ \varepsilon_n^2}\iint_{D \times D} \eta_{ \varepsilon_n}  (x-y ) |(u\circ T_n(x)-u\circ T_n(y))\cdot (T_n(x)-T_n(y))|\rho(x) \rho(y)\, \de x \de y
    \]
    for every $u \in C^\infty(D;\R^d) \cap \textnormal{Lip}(D;\R^d)$,
    where, for simplicity of notation, we have dropped the tilde.
    By slicing, we can rewrite the energy as
    \[
     \int_{\frac{D-D}{\varepsilon_n}} \Big (\int_{\Pi^\xi} F_n^{\xi,y}(  (u \circ T_n)^{\xi, y}  , \rho, (D \cap (D-\varepsilon_n \xi))_y^\xi) \, \de \mathcal{H}^{d-1}(y)  \Big )|\xi|\eta(\xi)\de\xi,
    \]
    where we recall that
    \begin{align*}
    F_n^{\xi,y} & (v ,\rho,(D \cap (D-\varepsilon_n \xi))_y^\xi)\\
    &= \frac{1}{\varepsilon_n^2} \int_{(D \cap (D-\varepsilon_n \xi))_y^\xi} |  v (t+\varepsilon_n ) -  v  (t) )\cdot (T_n^{\xi,y}(t+\varepsilon_n)-T_n^{\xi,t}(t))|\rho^{\xi,y}(t+\varepsilon_n)\rho^{\xi,y}(t) \, \de t.
    \end{align*}
 
    Set $ I_{n}  :=(D \cap (D-\varepsilon_n \xi))_y^\xi$. \ We claim that for a.e.~$\xi \in \R^{d} \setminus\{0\}$ and $\mathcal{H}^{d-1}$-a.e.~$y \in \Pi^{\xi}$ it holds
     \begin{align}
     \label{e:claimm}
    \lim_{n \to \infty} & \, F_n^{\xi,y}(  (u \circ T_n)^{\xi, y} ,\rho,  I_{n}  ) \\
    &
    = \lim_{n \to \infty} \frac{1}{\varepsilon_n}  \int_{ I_{n}  }|(u^{\xi,y}(t+\varepsilon_n)-u^{\xi,y}(t))\cdot \xi|\rho^{\xi,y}(t+\varepsilon_n)\rho^{\xi,y}(t) \, \de t\,. \nonumber
    \end{align}
    To prove~\eqref{e:claimm}, we prove that the following error term tends to zero:  
    \begin{align}
    \label{e:verylong}
        &\frac{1}{\varepsilon_n^2} \bigg | \int_{I_{n}}\Big (|(u^{\xi,y}(t+\varepsilon_n)-u^{\xi,y}(t))\cdot \varepsilon_n\xi| 
        \\
        &
        \quad- |((u\circ T_n)^{\xi,y}(t+\varepsilon_n )-(u\circ T_n)^{\xi,y}(t))\cdot (T_n^{\xi,y}(t+\varepsilon_n)-T_n^{\xi,y}(t))| \Big) \di t\bigg| \nonumber
        \\
        &
        \le \frac{1}{\varepsilon_n^2} \bigg | \int_{I_{n}}\Big (|(u^{\xi,y}(t+\varepsilon_n )-u^{\xi,y}(t))\cdot (\varepsilon_n\xi-(T_n^{\xi,y}(t+\varepsilon_n ) -T_n^{\xi,y}(t))| \nonumber
         \\
        &
        \quad+|(u^{\xi,y}(t+\varepsilon_n)-u^{\xi,y}(t))\cdot (T_n^{\xi,y}(t+\varepsilon_n) -T_n^{\xi,y}(t))| \nonumber
        \\
        &
        \quad- |((u\circ T_n)^{\xi,y}(t+\varepsilon_n)-(u\circ T_n)^{\xi,y}(t))\cdot (T_n^{\xi,y}(t+\varepsilon_n)-T_n^{\xi,y}(t))| \Big)\di t \bigg | \nonumber
        \\
        &
        \le \frac{1}{\varepsilon_n^2} \bigg | \int_{I_{n}}\Big (|(u^{\xi,y}(t+\varepsilon_n)-u^{\xi,y}(t))\cdot (\varepsilon_n\xi-(T_n^{\xi,y}(t+\varepsilon_n) -T_n^{\xi,y}(t))|  \nonumber
        \\
        &
        \quad +|(u^{\xi,y}(\varepsilon_n +t)-u^{\xi,y}(t)-((u \circ T_n)^{\xi,y}(t+\varepsilon_n)-(u\circ T_n)^{\xi,y}(t)))\cdot (T_n^{\xi,y}(t+\varepsilon_n) -T_n^{\xi,y}(t))| \Big) \di t \bigg |. \nonumber
        % &\le \frac{1}{\varepsilon_n^2} \bigg | \int_{I}\Big (|(u^{\xi,y}(t+\varepsilon_n)-u^{\xi,y}(t))\cdot (\varepsilon_n\xi-(T_n^{\xi,y}(t+\varepsilon_n) -T_n^{\xi,y}(t))| \\
        % &+|(u^{\xi,y}(\varepsilon_n +t)-u^{\xi,y}(t)-((u \circ T_n)^{\xi,y}(t+\varepsilon_n)-(u\circ T_n)^{\xi,y}(t)))\cdot (T_n^{\xi,y}(t+\varepsilon_n) -(y+(t+\varepsilon_n)\xi)|\\
        % &+|(u^{\xi,y}(\varepsilon_n +t)-u^{\xi,y}(t)-((u \circ T_n)^{\xi,y}(t+\varepsilon_n)-(u\circ T_n)^{\xi,y}(t)))\cdot (y+(t+\varepsilon_n)\xi-(y+t \xi))|\\
        % &+|(u^{\xi,y}(\varepsilon_n +t)-u^{\xi,y}(t)-((u \circ T_n)^{\xi,y}(t+\varepsilon_n)-(u\circ T_n)^{\xi,y}(t)))\cdot (T_n^{\xi,y}(t) -(y+t\xi)|\Big)\rho^{\xi,y}(t+\varepsilon_n)\rho^{\xi,y}(t)\bigg |
    \end{align}
    Since $u \in \text{Lip}(D;\R^d)$, we can bound the first term on the right-hand side of~\eqref{e:verylong} by
    \[
    \frac{|I_{n} |}{\varepsilon_n^2}  \text{Lip}(u)|\varepsilon_n\xi| \|Id-T_n\|_{L^\infty}\,,
    \]
    where $\text{Lip}(u)$ denotes the Lipschitz constant of~$u$.  After integration over $\frac{D-D}{\varepsilon_n}$ and  $\Pi^\xi$, this quantity converges to zero as $n \to \infty$ by \eqref{eq:id-tntozero}.
    To bound the second term on the right-hand side of~\eqref{e:verylong}, we observe that
\begin{align*}
    &|(u^{\xi,y}(\varepsilon_n +t)-u^{\xi,y}(t)-((u \circ T_n)^{\xi,y}(t+\varepsilon_n)-(u\circ T_n)^{\xi,y}(t)))\cdot (T_n^{\xi,y}(t+\varepsilon_n) -T_n^{\xi,y}(t))| \\
        &\le|(u^{\xi,y}(\varepsilon_n +t)-u^{\xi,y}(t)-((u \circ T_n)^{\xi,y}(t+\varepsilon_n)-(u\circ T_n)^{\xi,y}(t)))\cdot (T_n^{\xi,y}(t+\varepsilon_n) -y-(t+\varepsilon_n)\xi)|\\
        &+|(u^{\xi,y}(\varepsilon_n +t)-u^{\xi,y}(t)-((u \circ T_n)^{\xi,y}(t+\varepsilon_n)-(u\circ T_n)^{\xi,y}(t)))\cdot (y+(t+\varepsilon_n)\xi-(y+t \xi))|\\
        &+|(u^{\xi,y}(\varepsilon_n +t)-u^{\xi,y}(t)-((u \circ T_n)^{\xi,y}(t+\varepsilon_n)-(u\circ T_n)^{\xi,y}(t)))\cdot (T_n^{\xi,y}(t) -(y+t\xi))|.
\end{align*}
    Arguing as above,  each of these contributions tends to zero after integration over $\frac{D-D}{\varepsilon_n}$, $\Pi^\xi$, and $I_{n}$ in view of~\eqref{eq:id-tntozero} and of the Lipschitz continuity of~$u$.  Thus, we have shown that
    \begin{align}
    \label{e:claimm2}
    & \lim_{n \to \infty} \frac{1}{\varepsilon_n^2}   \int_{\frac{D - D}{\varepsilon_{n}}} \int_{\Pi^{\xi}} \bigg | \int_{I_{n}}\Big (|(u^{\xi,y}(t+\varepsilon_n)-u^{\xi,y}(t))\cdot \varepsilon_n\xi| 
        \\
        &
        - |((u\circ T_n)^{\xi,y}(t+\varepsilon_n )-(u\circ T_n)^{\xi,y}(t))\cdot (T_n^{\xi,y}(t+\varepsilon_n)-T_n^{\xi,t}(t))| \Big)\di t \bigg | \, \de \mathcal{H}^{d-1} (y) \, \eta(\xi) \di \xi = 0\,. \nonumber
    \end{align}
    Since $\rho$ is uniformly bounded from above and below (cf.~$(\rho2)$), \eqref{e:claimm2} implies~\eqref{e:claimm} for a.e.~$\xi \in \R^{d} \setminus\{0\}$ and $\mathcal{H}^{d-1}$-a.e.~$y \in \Pi^{\xi}$. 
    %In the following, to simplify the notation, we omit the factor $\rho^{\xi,y}(\varepsilon_n+\cdot)\rho^{\xi,y}(\cdot)$, this is justified when proving the upper bound, since by assumption it is uniformly bounded from above.Hence, the claim is proved.

    By the fundamental theorem of calculus, we have
    \begin{align*}
    \frac{1}{\varepsilon_n}  \int_{ I_{n} } & |(u^{\xi,y}(t+\varepsilon_n)-u^{\xi,y}(t))\cdot \xi|\rho^{\xi,y}(t+\varepsilon_n)\rho^{\xi,y}(t)\, \de t\\
    &=\frac{1}{\varepsilon_n}  \int_{ I_{n} }\Big |\int_t^{t+\varepsilon_n}\nabla u^{\xi,y}(\tau)\de \tau\cdot \xi\Big|\rho^{\xi,y}(t+\varepsilon_n)\rho^{\xi,y}(t)\, \de t.
    \end{align*}
    Recalling that $u \in C^\infty(D;\R^d)$ and $\rho \in C(D)$, we have, for every $t \in J \Subset D_y^\xi$,
    \[
    \lim_{n \to \infty} \fint_t^{t+\varepsilon_n}|\nabla u^{\xi,y}(\tau)\de \tau\cdot \xi|\rho^{\xi,y}(t+\varepsilon_n)\rho^{\xi,y}(t) =|\nabla u^{\xi,y}(t)\cdot \xi|(\rho^{\xi,y}(t))^2.
    \]
    Since $(\rho2)$ holds and  
    \[
    \frac{1}{\varepsilon_n} \Big|\int_t^{t+\varepsilon_n}\nabla u^{\xi,y}(\tau) \de \tau\cdot \xi\Big| \le \text{Lip}(u) |\xi|^2,
    \]
    we can apply the dominated convergence theorem and infer
    \begin{align*}
        &\lim_{n \to \infty}\frac{1}{\varepsilon_n} \int_{\frac{D-D}{\varepsilon_n}} \Big (\int_{\Pi^\xi}\int_{ I_{n}  }|(u^{\xi,y}(t+\varepsilon_n)-u^{\xi,y}(t))\cdot \xi|\rho^{\xi,y}(t+\varepsilon_n)\rho^{\xi,y}(t) \,\de t \de y\Big ) |\xi| \eta(\xi) \de \xi\\
        &\le \lim_{n \to \infty} \int_{ \frac{D-D}{\varepsilon_n} } \Big (\int_{\Pi^\xi}\int_{ I_{n}  }\fint_t^{t+\varepsilon_n}|\nabla u^{\xi,y}(\tau)\de \tau\cdot \xi|\rho^{\xi,y}(t+\varepsilon_n)\rho^{\xi,y}(t) \,\de t \de y\Big ) |\xi| \eta(\xi) \de \xi\\
        &=\int_{\R^d} \Big (\int_{\Pi^\xi}\int_{\Omega_y^\xi}|\nabla u^{\xi,y}(t)\cdot \xi|\rho^{\xi,y}(t)^2 \,\de t \de y\Big ) |\xi| \eta(\xi) \de \xi=TV_\eta(u;\rho^2).
    \end{align*}
     This concludes the proof of the limsup inequality~\eqref{e:limsup-ineq}.
\end{proof}

\section*{Acknowledgements}
This research has been supported by the Austrian Science Fund (FWF) through grants 10.55776/F65, 10.55776/Y1292, 10.55776/P35359, by the OeAD-WTZ project CZ04/2019 (M\v{S}M\\T\v{C}R 8J19AT013), by the University of Naples Federico II through the FRA Project ``ReSinApas'', and by the INdAM-GNAMPA project ``Sistemi multi-agente e replicatore: derivazione particellare e ottimizzazione'' CUP E53C25002010001. S.A. is member of Gruppo Nazionale per l'Analisi Matematica, la Probabilità e le loro Applicazioni (GNAMPA) of the Istituto Nazionale di Alta Matematica (INdAM).

%\bibliographystyle{siam}
%\bibliography{bibliography.bib}

\begin{thebibliography}{10}

\bibitem{AlmiDavoliKubinTasso}
{\sc S.~Almi, E.~Davoli, A.~Kubin, and E.~Tasso}, {\em On {D}e {G}iorgi's
  conjecture of nonlocal approximations for free-discontinuity problems: The
  symmetric gradient case}, Preprint,  (2024).

\bibitem{Ambrosio-Coscia-DalMaso}
{\sc L.~Ambrosio, A.~Coscia, and G.~Dal~Maso}, {\em Fine properties of
  functions with bounded deformation}, Arch. Rational Mech. Anal., 139 (1997),
  pp.~201--238.

\bibitem{AFP}
{\sc L.~Ambrosio, N.~Fusco, and D.~Pallara}, {\em Functions of bounded
  variation and free discontinuity problems}, Oxford Mathematical Monographs,
  The Clarendon Press, Oxford University Press, New York, 2000.

\bibitem{BachBraidesZeppieri}
{\sc A.~Bach, A.~Braides, and C.~Zeppieri}, {\em Quantitative analysis of
  finite-difference approximations of free-discontinuity problems}, Interfaces
  Free Bound., 22 (2020), pp.~317--381.

\bibitem{Bourdinetal}
{\sc B.~Bourdin, G.~Francfort, and J.-J. Marigo}, {\em The variational approach
  to fracture}, J. Elasticity, 91 (2008), pp.~5--148.

\bibitem{Braides-Caroccia}
{\sc A.~Braides and M.~Caroccia}, {\em Asymptotic behavior of the {D}irichlet
  energy on {P}oisson point clouds}, J. Nonlinear Sci., 33 (2023), pp.~Paper
  No. 80, 57.

\bibitem{Braides-Piatnitski}
{\sc A.~Braides and A.~Piatnitski}, {\em Homogenization of ferromagnetic
  energies on {P}oisson random sets in the plane}, Arch. Ration. Mech. Anal., 243
  (2022), pp.~433--458.

\bibitem{Bressonetal}
{\sc X.~Bresson, T.~Laurent, D.~Uminsky, and J.~von Brecht}, {\em Multiclass
  total variation clustering}, Advances in Neural Information Processing
  Systems, 26 (2013).

\bibitem{Caroccia}
{\sc M.~Caroccia}, {\em A compactness theorem for functions on {P}oisson point
  clouds}, Nonlinear Anal., 231 (2023), pp.~Paper No. 113032, 18.

\bibitem{Caroccia-Chambolle-Slepcev}
{\sc M.~Caroccia, A.~Chambolle, and D.~Slep\v{c}ev}, {\em Mumford-{S}hah
  functionals on graphs and their asymptotics}, Nonlinearity, 33 (2020),
  pp.~3846--3888.

\bibitem{Crismale-Scilla-Solombrino}
{\sc V.~Crismale, G.~Scilla, and F.~Solombrino}, {\em A derivation of
  {G}riffith functionals from discrete finite-difference models}, Calc. Var.
  Partial Differential Equations, 59 (2020), pp.~Paper No. 193, 46.

\bibitem{Cristoferi-Thorpe}
{\sc R.~Cristoferi and M.~Thorpe}, {\em Large data limit for a phase transition
  model with the p-{L}aplacian on point clouds}, European J. Appl. Math., 31
  (2020), pp.~185--231.

\bibitem{GarciaSellesetal}
{\sc D.~Garcia-Sell\'{e}s, E.~Carola, \`{O}.~Gratac\'{o}s, P.~Cabello, J.~Mu\~{n}oz, and
  O.~Ferrer}, {\em Geometrical characterization of fracture systems using point
  clouds and sefl-software: example of the añisclo anticline carbonate
  platform, central pyrenees}, Geologica acta, 22 (2024).

\bibitem{Trillos2020}
{\sc N.~Garc\'{\i}a~Trillos}, {\em Variational limits of {$k-NN$} graph-based
  functionals on data clouds}, SIAM J. Math. Data Sci., 1 (2019), pp.~93--120.

\bibitem{Trillos-Slepcev}
{\sc N.~Garc\'{\i}a~Trillos and D.~Slep\v{c}ev}, {\em Continuum limit of total
  variation on point clouds}, Arch. Ration. Mech. Anal., 220 (2016),
  pp.~193--241.

\bibitem{GTS}
{\sc N.~García~Trillos and D.~Slepčev}, {\em On the rate of convergence of
  empirical measures in {$\infty$}-transportation distance}, Canadian Journal
  of Mathematics, 67 (2015), p.~1358–1383.

\bibitem{Gobbino}
{\sc M.~Gobbino}, {\em Finite difference approximation of the mumford-shah
  functional}, Communications on Pure and Applied Mathematics, 51 (1998),
  pp.~197--228.

\bibitem{Gobbino-Mora}
{\sc M.~Gobbino and M.~Mora}, {\em Finite-difference approximation of
  free-discontinuity problems}, Proc. Roy. Soc. Edinburgh Sect. A, 131 (2001),
  pp.~567--595.

\bibitem{Rangapuram}
{\sc S.~Rangapuram and M.~Hein}, {\em Constrained 1-spectral clustering},
  Proceedings of the Fifteenth International Conference on Artificial
  Intelligence and Statistics, 22 (2012), pp.~1143--–1151.

\bibitem{Shi}
{\sc J.~Shi and J.~Malik}, {\em Normalized cuts and image segmentation},
  Pattern Analysis and Machine Intelligence, IEEE Transactions, 22 (2000),
  pp.~888–--905.

\bibitem{Slepcev-Thorpe}
{\sc D.~Slep\v{c}ev and M.~Thorpe}, {\em Analysis of {$p$}-{L}aplacian
  regularization in semisupervised learning}, SIAM J. Math. Anal., 51 (2019),
  pp.~2085--2120.

\bibitem{Sunetal}
{\sc J.~Sun, S.~Zhu, J.~Sun, J.~Zhou, Y.~Yao, Y.~Wang, J.~Zhang, B.~Zhou, and
  X.~Wang}, {\em A robust deep learning approach for rock discontinuity
  identification from large scale 3d point clouds}, Scientific Reports, 16
  (2025), pp.~637--650.

\bibitem{Thorpe-Theil}
{\sc M.~Thorpe and F.~Theil}, {\em Asymptotic analysis of the {G}inzburg-{L}andau
  functional on point clouds}, Proc. Roy. Soc. Edinburgh Sect. A, 149 (2019),
  pp.~387--427.

\end{thebibliography}

\end{document}